\documentclass[12pt]{article}
\usepackage{a4}
\usepackage[english]{babel}

\usepackage{amsthm}
\usepackage{amsmath}
\usepackage{amsfonts}
\usepackage{amssymb}
\usepackage{graphicx}
\usepackage{enumitem}
\setlist[enumerate,1]{label=(\arabic*)}
\setlist{noitemsep}

\newtheorem{theorem}{Theorem}

\newtheorem{lemma}[theorem]{Lemma}
\newtheorem{coro}[theorem]{Corollary}
\newtheorem{claim}{Claim}

\newcommand*{\claimproof}{Proof of the Claim}
\newenvironment{proofcl}[1][\claimproof]{\begin{proof}[#1]
}{\end{proof}}

\newcommand{\sm}{\setminus}
\newcommand{\FF}{\mathcal F}

\newcommand{\vl}{l\kern-0.035cm\char39\kern-0.03cm}

\title{Trestles in the squares of graphs}

\author{Adam Kabela\thanks{Affiliation: Faculty of Informatics, Masaryk University, Brno, Czech Republic.
E-mail: kabela@fi.muni.cz.
Former affiliation:
Faculty of Applied Sciences, University of West Bohemia, Pilsen, Czech Republic.}
\and
Jakub Teska\thanks{Affiliation: Department of Mathematics and European Centre of Excellence NTIS,
Faculty of Applied Sciences, University of West Bohemia, Pilsen, Czech Republic. E-mail: teska@kma.zcu.cz}}

\date{}

\begin{document}
\maketitle

\begin{abstract}
We show that the square of every connected $S(K_{1,4})$-free graph
satisfying a matching condition
has a $2$-connected spanning subgraph of maxi\-mum degree at most~$3$.
Furthermore,
we characterise trees whose square
has a $2$-connected spanning subgraph of maximum degree at most~$k$.
This generalises the results on $S(K_{1,3})$-free graphs of
Henry and Vogler (1985) and Harary and Schwenk (1971), respectively.

\vspace{3mm}
\end{abstract}


In this note, we continue the long-established and thorough study
of Hamiltonian properties of the squares of graphs
(for instance, see~\cite{N, CEF}).

We recall that the \emph{square} of a graph $G$ is
the graph on the same vertex set as~$G$ in which
two vertices are adjacent if and only if their distance in $G$ is
either $1$ or~$2$, and we let $G^2$ denote this graph.
We recall that a \emph{$k$-trestle}
(sometimes called $k$-covering)
is a $2$-connected spanning subgraph of maximum degree at most~$k$.
Clearly, $2$-trestles are Hamilton cycles;
and $k$-trestles are viewed as an extension of the concept of Hamiltonicity
(for instance, see~\cite{JKRS} and the references therein).

We let $S(K_{1,k})$ denote the graph obtained from $K_{1,k}$
by subdividing each of its edges once (see Figure~\ref{figSK13}). 
Clearly, the square of $S(K_{1,k+1})$ has no $k$-trestle.
%
We recall that Neuman~\cite{N} (and also Harary and Schwenk~\cite{HS})
showed that for trees,
being $S(K_{1,3})$-free (and having at least $3$ vertices)
is a necessary and sufficient condition 
in relation to Hamiltonicity of the square.
Later, Henry and Vogler~\cite{HW} showed that
this condition is sufficient for all graphs
(and this result was strengthened by
Abderrezzak, Flandrin and Ryj\'{a}\v{c}ek~\cite{AFR}
who studied additional properties of induced copies of $S(K_{1,3})$).

\begin{figure}[ht]
    \centering
    \includegraphics[scale=0.6]{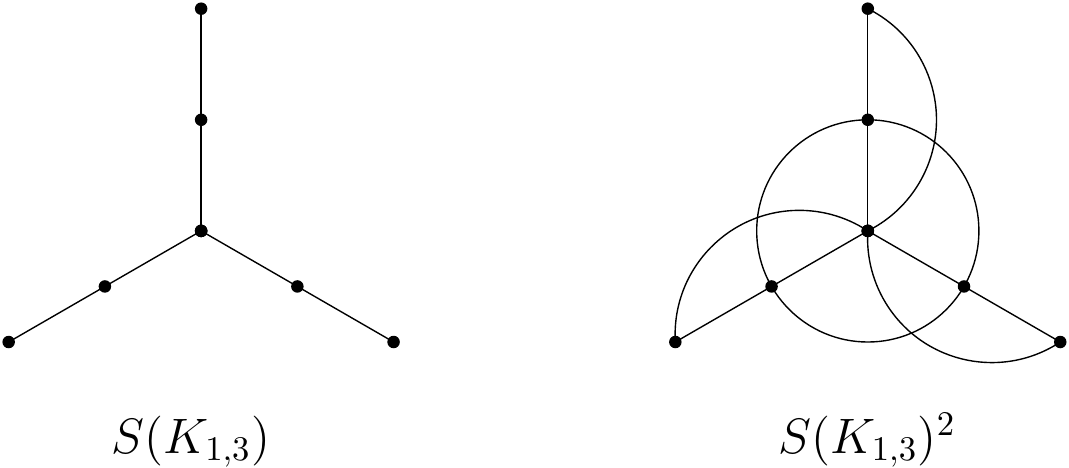}
    \caption{Graph $S(K_{1,3})$ and its square.}
    \label{figSK13}
\end{figure}

We investigate the squares of $S(K_{1,k+1})$-free graphs
and their $k$-trestles. For $k = 3$, we show the following.

\begin{theorem}
\label{main}
Let $G$ be a connected $S(K_{1,4})$-free graph (on at least~$3$ vertices).
Let $X$ be the set of all vertices $x$
such that $x$ is the centre of an induced copy of $S(K_{1,3})$ in $G$.
If $G$ has a matching of size $|X|$ whose every edge is incident with
precisely one vertex of $X$ and one vertex of $V(G) \sm X$,
then $G^2$ has a $3$-trestle such that
all non-matched vertices have degree $2$.
\end{theorem}

In case $G$ is $S(K_{1,3})$-free, the obtained trestle
is, in fact, a Hamilton cycle
(in particular, Theorem~\ref{main} can be viewed
as a generalization of the result of~\cite{HW}).

For the study of $k$-trestles in the squares of $S(K_{1,k+1})$-free graphs,
we shall need an extension of the matching condition of Theorem~\ref{main}.
To this end,
we recall that the \emph{symmetric orientation} of a graph
is the digraph obtained by replacing every edge
by a pair of antiparallel directed arcs,
and we consider particular assignments of integers
to arcs of the symmetric orientation.
Restricting to the squares of trees, we show the following
(which generalizes the result of~\cite{HS}).

\begin{theorem}
\label{tree}
Let $k$ be an integer greater than~$1$,
and $T$ be a tree (on at least~$3$ vertices),
and let $n(v)$ denote the number of non-leaves adjacent to vertex $v$ in~$T$.
Take the symmetric orientation of~$T$
and an assignment of non-negative integers to its arcs,
and let $i(v)$ denote the sum of the integers over all arcs ending in vertex~$v$
and $o(v)$ denote the sum over all arcs starting in~$v$.
The following statements are equivalent. 
\begin{enumerate}
\item
$T^2$ has a $k$-trestle.
\item
Every vertex $v$ of~$T$ satisfies $n(v) \leq k$,
and there exists a considered assignment such that
$i(v) = \max \{0, n(v) - 2 \}$ and
$o(v) \leq k - n(v)$
for every vertex $v$ of~$T$.
\item
$T^2$ has a $k$-trestle whose every vertex $v$ 
has degree $o(v) + \max \{2, n(v)\} $. 
\end{enumerate}

\end{theorem}

We remark that the assignment condition of statement (2) can be checked easily
(for instance, by using an auxiliary flow network).
For the case $k=2$ of Theorem~\ref{tree},
this condition is satisfied if and only if
every assigned integer is~$0$.
Considering the case $k=3$
and the set of all arcs whose assigned integer is positive,
we note that this set corresponds to a matching described in Theorem~\ref{main}
(in particular, the result of Theorem~\ref{main} is,
in some sense, sharp).

\begin{figure}[ht]
    \centering
    \includegraphics[scale=0.6]{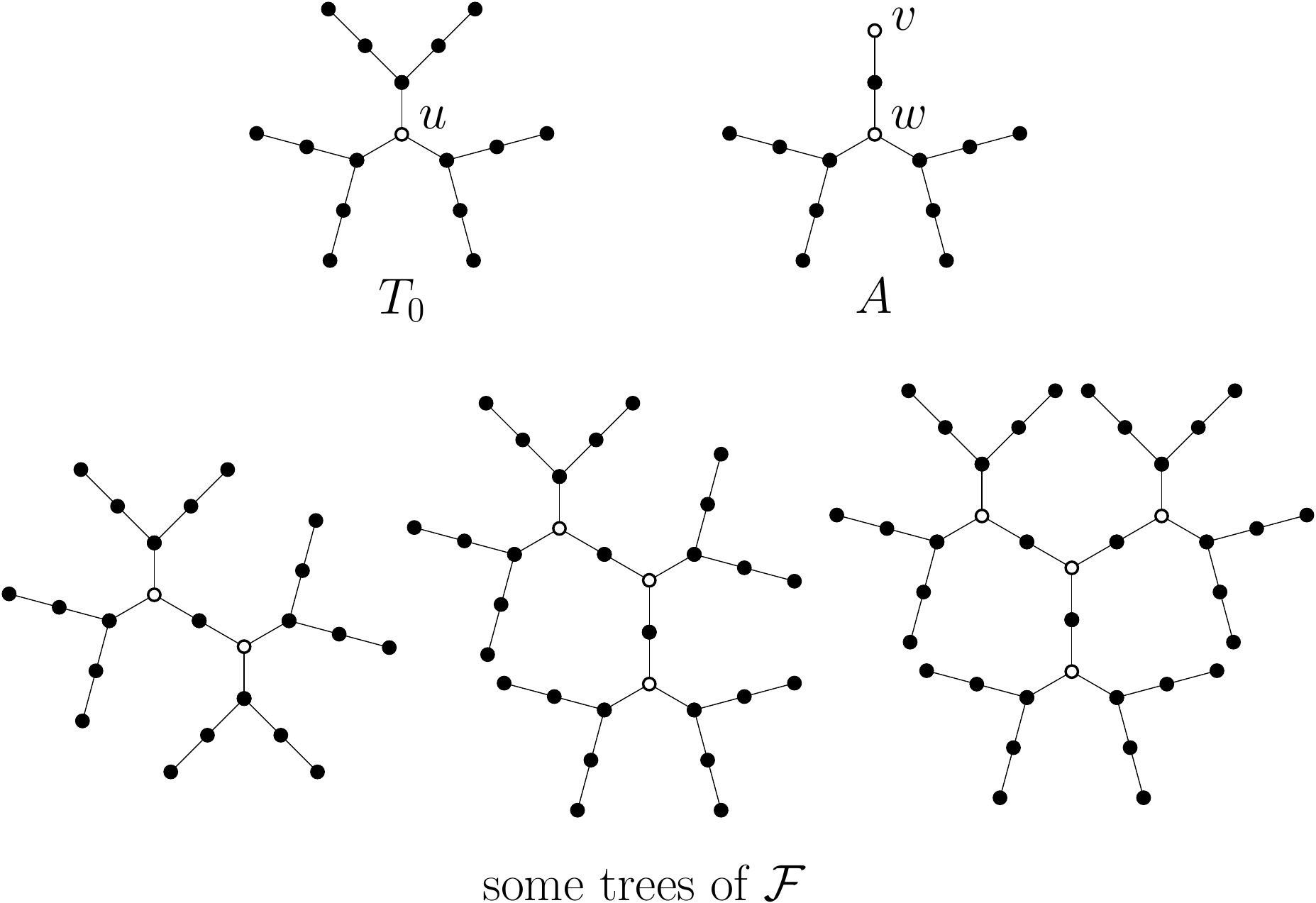}
    \caption{Trees $T_0$ and $A$ (top) and some examples of trees of $\FF$ (bottom).
    The special vertices are depicted as white.}
    \label{f:subtrees}
\end{figure}

In addition,
we note that the constraints preventing a tree from having a $3$-trestle in its square
can also be easily described in terms of subtrees.
We let $T_0$ and $A$ be the trees depicted in Figure~\ref{f:subtrees},
and we call vertex $u$ of $T_0$ and vertices $v$ and $w$ of $A$ \emph{special}.
We define an infinite family $\FF$ of trees (some of whose vertices are special) as follows.
A tree belongs to $\FF$ if and only if either it is $T_0$ or it can be obtained from a tree of $\FF$
by removing five of its vertices so that the resulting graph is a tree and its special vertex has degree $2$,
and by identifying this special vertex of the resulting graph with vertex $v$ of $A$.
We define the special vertices of the obtained tree in the natural way.
Some examples of trees of $\FF$ can be found depicted in Figure~\ref{f:subtrees}.
We use Theorems~\ref{main} and~\ref{tree}  
and the classical result of Hall~\cite{H} on matchings in bipartite graphs,
and we show the following.

\begin{coro}\label{coro}
Let $T$ be a tree (on at least~$3$ vertices)
and $\FF$ be the family of trees defined above.
Then $T^2$ has a $3$-trestle if and only if
$T$ is $S(K_{1,4})$-free and
for every subtree of $T$ isomorphic to a tree of $\FF$,
at least one special vertex of the subtree has degree greater than $3$ in $T$.
\end{coro}

The proofs of Theorems~\ref{main} and~\ref{tree} and Corollary~\ref{coro} 
are included below.
In~the proof of Theorem~\ref{main}, we extend the idea of~\cite{HW}.
We shall use
the result of Fleischner~\cite{F} on the squares of $2$-connected graphs,
and the following lemma.
We recall that a \emph{linear forest} is a graph whose every component is a path
(we view a vertex of degree $0$ as a trivial path),
and the \emph{independence} of a graph (digraph)
is the size of its maximum independent set.

\begin{lemma}
\label{linearForest}
For every independent set $I$ of a graph of independence $k$,
there exists a spanning linear forest
such that it has at most $k$ components each containing at most one vertex of $I$
and every vertex of $I$ is of degree at most~$1$.  
\end{lemma}

We view Lemma~\ref{linearForest}
as a corollary of the following result of Gallai and Milgram~\cite[Satz~3.1]{GM}.

\begin{theorem}
\label{pathCover}
For every digraph of independence $k$, its vertex set
can be covered by at most $k$ vertex-disjoint paths (possibly trivial).
\end{theorem}

\begin{proof}[Proof of Lemma~\ref{linearForest}]
We let $G$ denote the given graph, 
and $D$ be a digraph obtained from $G$ by replacing every edge
with a directed arc as follows.
For every vertex of $I$, all arcs incident with this vertex are oriented towards it;
and the orientation of the remaining arcs is chosen arbitrarily.

We consider a path cover of $D$ given by Theorem~\ref{pathCover},
and we note that it consists of at most $k$ paths and every vertex of $I$ is an end of some of them.
We conclude that this path cover of $D$ translates into a desired subgraph of $G$.
\end{proof}

Finally, we prove Theorems~\ref{main} and~\ref{tree} and Corollary~\ref{coro}.

\begin{proof}[Proof of Theorem~\ref{main}]
We note that the statement is satisfied for graphs on at most four vertices.
We suppose that it is satisfied for graphs which have fewer vertices than $G$,
and we show it for $G$.

We recall that if $G$ is $2$-connected, then $G^2$ is Hamiltonian by~\cite{F}.
Furthermore, we note that if $G$ is a path, then $G^2$ is Hamiltonian.
Consequently, we can assume that $G$ has a cutvertex, say $c$, of degree at least $3$.

We let $N(c)$ denote the set of all neighbours of $c$ in $G$.
We consider a matching satisfying the assumptions of the theorem,
and we let $M_0$ be the set of all edges of the matching
whose both ends belong to $N(c) \cup \{ c \}$,
and $M$ be the set of all remaining edges of the matching.
We let $T$ be a spanning tree of $G$
containing all edges incident with $c$ and all edges of~$M$
(clearly, such a spanning tree exists since $G$ is connected
and the fixed edges give a forest).
We note that the graph $T - c$ has at least three components
(since $c$ is of degree at least $3$ in $G$),
and every component of $T - c$
contains precisely one vertex of $N(c)$.

In addition,
we can assume that at least one component of $T - c$ is non-trivial 
(otherwise $G^2$ is a complete graph, and thus Hamiltonian).
We consider all non-trivial components of $T - c$,
and we let $V_1, \dots, V_k$ denote their vertex sets,
and $u_i$ denote the neighbour of $c$ in $V_i$.

For every $i=1, \dots, k$,
we consider the subgraph of $G$ induced by $V_i \cup \{c\}$,
and we extend it by adding an auxiliary vertex $y_i$ and the edge $c y_i$;
and we let $H_i$ denote the resulting graph.
(We view $y_i$ as an extra vertex not belonging to~$V(G)$.
On the other hand, $H_i$ can be viewed as an induced subgraph of $G$
since $c$ is a cutvertex in $G$.) 
In addition, we let $M_i$ be the restriction of $M$
consisting of all edges of $M$
whose one end is a centre of induced $S(K_{1,3})$ in $H_i$;
and if $u_i$ is a centre of induced $S(K_{1,3})$ in $H_i$
and the vertex matched with $u_i$ does not belong to $V_i$,
then we extend $M_i$ by adding the edge $cu_i$.
We show the following.
\begin{claim} 
\label{c1}
For every $i=1, \dots, k$,
the graph $H_i^2$ has a $3$-trestle
such that every vertex which is not covered by $M_i$ has degree~$2$.
\end{claim}
\begin{proofcl}[Proof of Claim~\ref{c1}]
We recall that $H_i$ can be viewed an induced subgraph of $G$,
and hence if a vertex is a centre of induced $S(K_{1,3})$ in $H_i$,
then it is a centre of induced $S(K_{1,3})$ in~$G$
(in particular, vertex $y_i$ has degree $1$ and $c$ has degree $2$ in $H_i$
and so none of them is a centre in $H_i$).
We observe that $H_i$ considered with the matching $M_i$
satisfies the assumptions of the theorem.
Since $H_i$ has fewer vertices than $G$,
we conclude that
$H_i^2$ has a desired $3$-trestle by the induction hypothesis.
\end{proofcl}

For every $i=1, \dots, k$,
we let $Z_i$ be a $3$-trestle of $H_i^2$ given by Claim~\ref{c1}.
We note that $Z_i$ contains the edges $cy_i$ and $u_iy_i$
(since $Z_i$ is $2$-connected and $y_i$ is only adjacent to $c$ and $u_i$ in $H_i^2$).
Also, we let $O_i$ denote the set of all vertices of $V_i$
which are adjacent to $c$ or $y_i$ in $Z_i$.
We note that $3 \geq |O_i| \geq 2$ (in particular, $u_i$ belongs to~$O_i$).
In addition, we let $R_i$ denote the graph $Z_i - \{c, y_i \}$.
We shall use the graphs $R_i$ for constructing a $3$-trestle of~$G^2$
(we view $O_i$ as the entry points of~$R_i$).
Preparing to show the $2$-connectivity of the desired construction,
we observe the following.
\begin{claim} 
\label{c2}
Let $r$ be a vertex of $R_i$ (where $1 \leq i \leq k$).
Then in the graph $R_i - r$,
every vertex is joint by a path to a vertex of $O_i \sm \{r\}$
(possibly a trivial path if the vertex itself belongs to $O_i \sm \{r\}$).
\end{claim}
\begin{proofcl}[Proof of Claim~\ref{c2}]
We note that every vertex of $V_i \sm \{r\}$
is joint by a path to $c$ in the graph  $Z_i - r$
(since $Z_i$ is $2$-connected), and 
for every vertex of $V_i \sm (O_i \cup \{r \})$
this path goes through a vertex of $O_i \sm \{r\}$.
The existence of a desired path in $R_i - r$ follows.
\end{proofcl}
We let $w_i$ denote a vertex of $O_i$ distinct from $u_i$.
We recall that $3 \geq |O_i| \geq 2$,
and so $O_i$ consists of $u_i$ and $w_i$ and possibly one additional vertex.

For every $i=1, \dots, k$ such that $|O_i| = 3$, we define edge $e_i$ as follows.
We observe that if $|O_i| = 3$, then $c$ is of degree $3$ in $Z_i$;
and by Claim~\ref{c1} we get that $c$ is incident with an edge of $M_i$,
and it is, in fact, the edge $cu_i$.
In particular, $u_i$ is a centre of induced $S(K_{1,3})$ in $H_i$,
and the vertex originally matched with $u_i$ does not belong to $V_i$
(the vertices are joint by an edge of $M_0$);
we say that the vertex is \emph{engaged} with $u_i$.
We let $e_i$ denote the edge joining the third vertex of $O_i$
to the engaged vertex
(possibly to $c$ if $c$ is engaged with $u_i$).

We note that no engaged vertex is a centre of induced $S(K_{1,3})$ in $G$,
and each engaged vertex belongs to $N(c) \cup \{c \}$
and is incident with only one of the edges $e_i$,
and that $e_i$ is an edge of~$G^2$.

We let $A$ be the subgraph of $G^2$
induced by $N(c) \cup \{w_1, \dots, w_k \}$. 
Furthermore, if $c$ is a centre of induced $S(K_{1,3})$ in $G$,
then we let $a$ denote the vertex matched with $c$ in~$M_0$.
We shall need particular paths in $A$ for interconnecting the graphs~$R_i$.
To this end, we show the following.

\begin{claim} 
\label{c3}
The graph $A$ has a spanning linear forest with the following properties.
\begin{enumerate}
\item
It contains all edges $u_i w_i$.
\item
It has at most three components; and if it has precisely three components,
then $c$ is a centre of induced $S(K_{1,3})$ in $G$.
\item
Every component has at least two vertices
and it has an end belonging to~$N(c)$;
and if there are precisely three components,
then one of these ends is the vertex $a$.

\end{enumerate}
\end{claim}
\begin{proofcl}[Proof of Claim~\ref{c3}]
We let $W$ denote the set consisting of
all vertices $u_i, w_i$ such that $w_i$ does not belong to $N(c)$.
We can assume that $W$ is non-empty
(otherwise $A$ is a complete graph and the claim follows).
We let $B_1, B_2$ denote the subgraph induced by $W$ in $G$, $G^2$, respectively 
(clearly, $B_2$ is an induced subgraph of $A$).

First, we contract all edges $u_i w_i$ of $B_1$,
and we let $C$ denote the resulting graph
and $c_i$ denote its vertex corresponding to $u_i w_i$
(for each edge $u_i w_i$).
We note that if $C$ is of independence $3$, then
$c$ is a centre of induced $S(K_{1,3})$ in~$G$.
Similarly, we get that $C$ is of independence at most $3$
(since $G$ is $S(K_{1,4})$-free).
We consider a spanning linear forest of $C$ given by Lemma~\ref{linearForest}.
In particular, if $a$ exists and belongs to $W$ (say $a$ is labeled as $u_i$), then
we consider a forest such that the corresponding vertex $c_i$ is of degree at most $1$;
such a forest can be obtained by applying Lemma~\ref{linearForest}
to graph $C$ and independent set $\{ c_i \}$.

Next, 
we expand each of the path-components of the forest to a path in $B_2$ as follows.
We observe that if $c_i$ is adjacent to $c_j$, then
each vertex of $\{ u_i, w_i \}$ is adjacent to at least
one of $\{ u_j, w_j \}$ in~$B_2$. 
For every path-component,
we replace each vertex $c_i$ by the pair $u_i w_i$ (or $w_i u_i$)
so that the resulting sequence is a path in $B_2$ whose end belongs to $N(c)$;
in particular, if $a$ exists and belongs to $W$, then
we can ensure that $a$ is an end of one of the paths
(by starting with the corresponding vertex $c_i$
and replacing it with the pair $a w_i$).
We let $F_{B_2}$ denote the resulting spanning linear forest of $B_2$.

Finally, we consider the vertices of $V(A) \sm W$.
If this set is empty, then we observe that $F_{B_2}$ satisfies the claim.
Otherwise, we use the fact that all vertices of $N(c)$
are pairwise adjacent in $A$,
and we take a path $P$ consisting of all vertices of $V(A) \sm W$ such that
it contains all remaining edges $u_i w_i$ (which are not included in $F_{B_2}$);
furthermore if $a$ exists and belongs to $V(A) \sm W$,
then we take $P$ so that $a$ is an end of $P$.
We extend $F_{B_2}$ by adding $P$ as follows.
We choose a path-component of $F_{B_2}$
and its end belonging to $N(c)$ and an end of $P$,
and we join the paths by adding the edge connecting the chosen ends;
in particular, if $F_{B_2}$ has three components, then 
we ensure that $a$ is an end of a path-component in the resulting graph
(by choosing a component not containing $a$ and an end of $P$ distinct from $a$).
We conclude that the resulting graph is a spanning linear forest of $A$
satisfying properties (1), (2) and (3).
\end{proofcl}

We let $F$ be a forest given by Claim~\ref{c3},
and we extend it as follows.
We recall that each path-component of $F$ 
is non-trivial by property~(3), that is, it has two ends.
We let $L$ be the set given by property~(3)
consisting of precisely one end of each path,
and $L'$ be the set of the other ends
(clearly, $|L|$ is equal to the number of components of $F$).
We enhance $F$ by adding vertex $c$
and adding all edges joining $c$ to the vertices of $L'$,
and we finish the extension by discussing three cases based on $|L|$.
\begin{itemize}
\item
If $|L| = 1$, then we add the edge joining $c$ to the vertex of $L$ 
(and we note that the obtained graph is a cycle).
\item
If $|L| = 2$, then we add the edge joining the two vertices of $L$
(obtaining a cycle).
\item
Otherwise, we have $|L| = 3$ (by property (2) of Claim~\ref{c3}).
We use that $a$ belongs to $L$ by property (3),
and we add the two edges joining $a$ to the vertices of $L \sm \{a \}$
(and we obtain a graph consisting of three internally disjoint paths
from $a$ to $c$).
\end{itemize}
We let $\Theta$ denote the resulting extension of $F$,
and we note that $\Theta$ is a $2$-connected subgraph of $G^2$
on the vertex set $V(A) \cup \{c \}$.
Finally, we use property (1) of Claim~\ref{c3} and expand $\Theta$ as follows.
For every $i=1, \dots, k$ in sequence, we remove the edge $u_i w_i$ (of $\Theta$)
and take the union of the graph on hand with the graph $R_i$,
and if $|O_i| = 3$, then we add the edge~$e_i$;
and we let $Z$ denote the resulting graph.

We note that $Z$ is a spanning subgraph of $G^2$.
In order to verify that $Z$ is $2$-connected,
we consider removing an arbitrary vertex of $Z$
and we observe that the obtained graph is connected 
(in particular, we view the construction as interconnecting the graphs $R_i$
and use Claim~\ref{c2} and the fact that $\Theta$ is $2$-connected).
Lastly, we discuss the degrees of vertices in $Z$.
We note that vertex $c$ is of degree at most $3$ in $Z$;
and if it has degree $3$, then either
it is engaged with some vertex $u_i$
or it is a centre of induced $S(K_{1,3})$ in $G$.
Also,
we note that vertex $a$ (if it exists) is of degree at most $3$ in $Z$.
We consider the vertices (distinct from $c$ and $a$)
belonging to none of the graphs $R_i$,
and we note that each such vertex has degree~$3$ in $Z$ if engaged,
and degree $2$ otherwise
(since it is of degree $2$ in $\Theta$).
It remains to discuss the vertices of the graphs $R_i$ (distinct from $a$).
We note that if such a vertex does not belong to $O_i$,
then it is not engaged and its degree in $Z$ is equal to its degree in~$Z_i$.
Furthermore, 
for every vertex of $O_i$ at least one edge of $Z_i$ is removed in the construction
and one edge of $\Theta$ is added;
and if such a vertex is engaged, then it is of the degree $2$ in $Z_i$ (by Claim~\ref{c1}) 
and one of the edges $e_i$ is added.
We conclude that every vertex of degree $3$ in $Z$ is covered by $M_0$ or $M$. 
It follows that $Z$ is a desired $3$-trestle of $G^2$.
\end{proof}

\begin{proof}[Proof of Theorem~\ref{tree}]
We show that (1) implies~(2).
We let $Z$ denote a given $k$-trestle in $T^2$.
We consider an arbitrary vertex, say $x$, of $T$
and let $U$ denote the set of all its neighbours in $T$,
and $N$ denote the graph induced by $U$ in $Z$.
Since $T$ is a tree and $Z$ is a $k$-trestle of $T^2$,
we have that $n(x) \leq k$ and that $N$ is connected
(we view trivial graph as connected).

In particular, $N$ has at least $|U|-1$ edges,
that is, the sum of degrees of its vertices
is at least $2|U|-2$.
We consider the arcs from $U$ to $x$ in the symmetric orientation of~$T$,
and we note that there exists an assignment with the following properties
(since $2|U|-2 \geq |U| + n(x) - 2$).
\begin{itemize}
\item
If $x$ is a leaf in $T$ (that is, $N$ is trivial), then
the arc from $U$ to $x$ is assigned~$0$.
\item
Otherwise,
for every vertex $u$ of $U$, the arc $ux$ 
is assigned a non-negative integer smaller than the degree of $u$ in $N$;
and the sum of the assigned values is equal to $\max \{0, n(x) - 2 \}$.
\end{itemize}

We consider such assignment for every vertex of $T$,
and we take the union of all these assignments.
We conclude that the resulting assignment satisfies
$i(v) = \max \{0, n(v) - 2 \}$
and
$o(v) \leq k - n(v)$ for every vertex $v$ of $T$.

We show that (2) implies (3)
by examining a hypothetical minimal counterexample to the implication.
We consider the smallest $k$ for which a counterexample exists,
and we let $T$ be a counterexample on the smallest number of vertices for this $k$.

By Theorem~\ref{main} and by the choice of $T$, we can assume that
there exists a vertex $x$ such that $n(x) \geq 3$,
and we let
$u_1, \dots, u_{n(x)}, u_{n(x)+1}, \dots, u_{\ell}$ denote its neighbours in $T$
so that $u_1, \dots, u_{n(x)}$ are non-leaves.
We let $a(u_j x)$, $a(x u_j)$ denote the integer
assigned to the arc $u_j x$, $x u_j$, respectively
(for every $j = 1, \dots, \ell$).
For every $j = 1, \dots, n(x)$,
we take the component of $T - x$ containing $u_j$
and extend it by adding vertices $x$ and $y_j$ and the edges $u_j x$ and $x y_j$;
and we let $T_j$ denote the resulting tree. 

We consider $T_j$ with the corresponding restriction of the assignment
for its symmetric orientation
(assigning $0$ to the arcs $u_j x$, $x y_j$ and $y_j x$).
By the choice of~$T$, 
there exists a $k$-trestle of $T_j^2$
and the degrees of its vertices correspond to the restricted assignment;
and we let $Z_j$ denote such $k$-trestle.
In particular, $y_j$ has degree $2$ 
(it is adjacent to $x$ and $u_j$)
and $x$ has degree $a(x u_j) + 2$
and $u_j$ has degree $o(u_j) - a(u_j x) + \max \{2, n(u_j)\}$
(where $o(u_j)$ refers to the original assignment in~$T$).

Also, we consider a graph 
on the vertex set $\{ u_1, \dots, u_{\ell} \}$
such that for every $j = 1, \dots, \ell$,
the vertex $u_j$ has degree $a(u_j x) + 1$ if $j \leq n(x)$,
and degree $a(u_j x) + 2$ otherwise;
and we note that the sum of degrees is $2\ell - 2$
(since $i(x) = n(x) - 2$).
In particular, there exists such a graph which is a tree, say $T_U$.

Finally, we take the union of all graphs $Z_j - y_j$ 
(in this union,
$x$ is of degree $o(x) + n(x)$, and 
for every $j = 1, \dots, n(x)$,
the degree $u_j$ is $o(u_j) - a(u_j x) + \max \{2, n(u_j)\} - 1$),
and we extend this graph by adding the vertices $u_{n(x)+1}, \dots, u_{\ell}$ 
and adding all edges of $T_U$.
We note that the resulting graph is a $k$-trestle of $T^2$
contradicting the choice of $T$.

Clearly, (3) implies (1) which concludes the proof.
\end{proof}

\begin{proof}[Proof of Corollary~\ref{coro}]
For the sake of simplicity, we say that 
a vertex is red if it is a centre of $S(K_{1,3})$ in $T$,
and it is black otherwise.

First, we suppose that $T^2$ has a $3$-trestle.
We apply Theorem~\ref{tree} and note that statement (2) yields that
$T$ is $S(K_{1,4})$-free and $T$ has a matching covering all red vertices
such that each edge of the matching is incident with precisely one black vertex. 
We consider an arbitrary subtree $S$ of $T$ isomorphic to a tree of $\FF$,
and we note that its special vertices are coloured red in $T$ 
and some of them has to be matched with a vertex outside $S$
(since there are not sufficiently many black neighbours in $S$).
Thus, the degree of this special vertex is greater than $3$ in~$T$.

Next, we suppose that $T$ is $S(K_{1,4})$-free and the condition on subtrees is satisfied.
For the sake of a contradiction,
we suppose that $T^2$ has no $3$-trestle.
We consider the spanning bipartite subgraph $B$ of $T$ given by the red-black colouring
(that is, by deleting each edge whose ends have the same colour).
By Theorem~\ref{main}, we have that $B$ has no matching covering all red vertices.
We consider a minimal set $R$ of red vertices
violating the condition of Hall's theorem,
and we let $N(R)$ be the set of all black vertices adjacent to a vertex of $R$ in $B$.
We observe that the minimality implies that the vertices of $R$
are precisely the special vertices of a subtree of $T$ isomorphic to a tree of $\FF$.
Since $|R| > |N(R)|$ and $T$ is $S(K_{1,4})$-free,
we conclude that no vertex of $R$ has degree greater than $3$ in~$T$,
a~contradiction.
%
\end{proof}
\section*{Acknowledgements}
We thank the anonymous referees for their helpful comments and suggestions.
The work of the first author was supported by 
the MUNI Award in Science and Humanities of the Grant Agency of Masaryk University.
The work of the second author was supported by
the project GA20-09525S of the Czech Science Foundation.

\end{document}